\documentclass[10 pt]{article}
\usepackage[header, margin]{classic}
\usepackage{quiver}
\usepackage{float}
\usetikzlibrary{automata,arrows,positioning,calc}
\usepackage{amsrefs}

\usepackage{bm}
\usepackage{xcolor}

\begin{document} 

\title{Quantitative gap universality for Wigner matrices} 

\vspace{1cm}
\noindent

\begin{minipage}[b]{0.3\textwidth}
\hspace{3cm}
 
 \end{minipage}
 \begin{minipage}[b]{0.3\textwidth}
 \author{Albert Zhang}

\address{Courant Institute, \\
		New York University \\
acz271@nyu.edu}
 \end{minipage}

\begin{minipage}[b]{0.3\textwidth}

 \end{minipage}

\begin{abstract}
		We obtain  the explicit rate of convergence $N^{-\frac{1}{2}+\varepsilon}$ for the gaps between eigenvalues of generalized Wigner matrices in the bulk of the spectrum,   for distributions of matrix entries possibly atomic and supported on enough points,  quantifying the universality result by Erd{\H o}s and Yau \cite{ErYa15}.  The proof proceeds by a Green function comparison, coupled with the relaxation estimate  from \cite{Bo22}. In particular, we extend the 4 moment matching method \cite{TaVu11} to arbitrary moments, allowing to compare resolvents down to the submicroscopic scale $N^{- \frac{3}{2} + \varepsilon}$. 

This method also gives universality of the smallest gaps between eigenvalues for the  Hermitian symmetry class,  providing a universal,  optimal separation of eigenvalues for discrete random matrices with  entries supported on $\Omega(1)$ points. 
\end{abstract}

\tableofcontents

\section{Introduction}

\subsection{Quantitative universality and extreme statistics in random matrix theory}

In the past two decades, significant progress on the universality of gaps between eigenvalues of (generalized) Wigner matrices has been made via the Erd{\H o}s-Schlein-Yau dynamical approach \cites{ErScYa11, ErYa15}, combined with the four moment theorem of Tao and Vu \cite{TaVu11}. See for example \cite{ErYa15} for results in the bulk of the spectrum, and \cite{BoErYa14} for results in the edge of the spectrum,  with non-explicit polynomial rates of convergence $N^{-c}$.   
Quantitative universality, on the other hand, has only been well studied at the edge. In \cite{Bo22}*{Theorem 1.5} the first explicit rate of convergence regarding the largest eigenvalue of a generalized Wigner matrix was shown. In particular,  if $\lambda_N$ denotes the largest eigenvalue  and $\t{TW}$  the Tracy-Widom distribution, we have for any $c > 0$ and large enough $N$ that
\begin{equation}\label{Bo-TW-result}
		d_{\t K}(N^{2/3} (\la_N - 2), \t{TW}) \leq N^{- \frac{2}{9} + c},
\end{equation}
where the Kolmogorov distance is defined in (\ref{eqn:Kolmogorov}). This result was then substantially improved to the almost optimal rate  $N^{- \frac{1}{3} + c}$ in \cites{ScXu22,ScXu23}.
One step for the proof of (\ref{Bo-TW-result}) is to apply a moment-matching or Green function comparison theorem in order to replace $H_N$ by a Gaussian divisible ensemble (GDE),
\[
		\tilde H_t := e^{-t/2} \tilde H_0 + (1 - e^{-t})^{1/2} U,
\]
where $\tilde H_0$ is a generalized Wigner matrix and $U$ is an independent Gaussian orthogonal ensemble (GOE). Then an edge relaxation estimate,  see e.g. \cite{Bo22}*{Theorem 2.8}, is used to replace the gaps of GDE with those of GOE. In a last step, existing quantitative estimates for the rate of convergence for GOE to Tracy-Widom are imported from, say, \cite{JoMa12}. 

We adapt and extend this strategy to the gaps in the bulk of the spectrum, for which we obtain the following estimate on the Kolmogorov distance:
\begin{equation}\label{intro-rate}
		d_{\t K}(N (\la_{k+1} - \la_k), N(\mu_{k+1} - \mu_k)) \leq N^{-\frac{1}{2} + \varepsilon},
\end{equation}
where $(\la_k)$ and $(\mu_k)$ are the ordered eigenvalues of a generalized Wigner matrix and GOE, respectively. We also show a universality result regarding the smallest gap for a complex Wigner matrix
\begin{equation}\label{eqn:min}
				d_{\t K} \l(\min_{\a N \leq k \leq (1 - \a) N} N^{ \frac{4}{3} } (\la_{k+1} - \la_k), \min_{\a N \leq k \leq (1 - \a) N} N^{ \frac{4}{3} } ( \mu_{k+1} - \mu_k) \r)={\rm o}(1),
\end{equation}
where the matrix entries  may have a discrete support, both for (\ref{intro-rate}) and (\ref{eqn:min}).

The bound (\ref{intro-rate}) provides the first explicit speed for bulk universality.  Regarding the optimal separation of eigenvalues (\ref{eqn:min}),
the previous result was only for smallest gaps of Wigner matrices whose entries were smooth on a sufficient scale \cite{Bo22}*{Corollary 1.3}. This result from \cite{Bo22} relies on the reverse heat flow,  see for example \cite{Bo22}*{Lemma 4.1}, as a replacement for the moment matching method, excluding  Wigner matrices whose entries are atomic.  

One caveat is that for both (\ref{intro-rate}) and (\ref{eqn:min}),  the entries of our Wigner matrices are required to be supported on at least $p$ points for a large enough $p$,  e.g.  $p = p(\varepsilon) < \infty$ grows as $\varepsilon \downarrow 0$ in  (\ref{intro-rate}) .  
This constraint arises due the need to extend the four moment theorem of Tao and Vu \cite{TaVu11} by allowing moments to match up to the $p$-th order, for some large $p = p(\varepsilon)$ not depending on $N$, see Lemma~\ref{GDE-moment-matching}. This enables us to obtain better error estimates in our Green function comparison therefore achieving an explicit rate of convergence for the gaps (\ref{intro-rate}).

Finally,  as mentioned before we note that the edge result (\ref{Bo-TW-result}) has been improved in \cites{ScXu22,ScXu23} to $N^{- \frac{1}{3} + c}$, but the methods we use are closer to those in \cite{Bo22} as we do not rely on 
 long-time comparison of the Green functions, a method particularly efficient at the edge of the spectrum but not in the bulk.  Also, for analogous results regarding the largest eigenvalue of sample covariance matrices see  \cites{ScXu23b,Wa24}, and for sparse random matrices see \cite{BucScXu25}. \\

The study of extreme gaps between random spectra started with integrable models. Vinson \cite{Vi01} showed that the smallest gap between eigenvalues of the $N \x N$ circular unitary ensemble (CUE), when rescaled by $N^{4/3}$, has limiting density $3x^2 e^{-x^2}$ as $N$ goes to $\infty$; similar results for Hermitian invariant ensembles (including the Gaussian Unitary Ensemble, GUE) were obtained in his thesis. 
Analogous results hold for the smallest gap for general translation invariant determinantal point processes in large boxes \cite{So05}.  
Smallest gaps statistics were extended to other temperatures, with the smallest gaps of the circular $\b$-ensemble \cite{FeWe21} with integer $\beta$ and the Gaussian orthogonal and symplectic ensembles \cites{FeTiWe19}*{FeLiYa24}.  
The question of largest gaps between nearest eigenvalues was raised in \cite{Di03} and solved at leading order in \cite{BABo13}, and up to the limiting distribution in \cite{FeWe18},  both for CUE and GUE.

Regarding universality of extreme gaps,  some of the results mentioned above on largest and smallest gaps for the CUE, GUE  extend to multimatrix  ensembles \cite{FiGu16}.  For Wigner matrices,  universality of the largest gaps 
was obtained in \cite{LaLoMa20} for any matrix entries,  while \cite{Bo22} showed universal statistics of both largest and smallest gaps for entries satisfying a smoothness condition, see (\ref{c:smooth}).  For symmetric random matrices with possibly discrete entries, 
suboptimal  bounds for separation between eigenvalues have been derived in \cites{LoLu21,NgTaVu17}. 
Motivations for studying the extreme eigenvalues' gaps statistics come from numerical linear algebra \cites{BABo13,DeTr19}, conjectures in analytic number theory \cites{BABo13,BuMi18}, algorithmic number theory \cite{BoHiKe15}, and quantum chaos in the complementary Poissonian regime \cite{BlBoRaRu17}.

Other extreme statistics are of interest in random matrix theory.  Notably 
the work of Fyodorov, Hiary, and Keating (FHK) \cite{FyoHiaKea12} conjectures the maximum of the characteristic polynomial of random matrices and relates it to the maximum of the Riemann zeta function on typical intervals of the critical line.   These random matrix extreme statistics are now well understood in the case of integrable, circular ensembles such as  \cites{ChMaNa18,PaqZei25},  and at the leading order for universal Wigner and $\beta$-ensembles  \cite{BoLoZe25}.  High moment matching of random matrix ensembles as developed in our work may help refine universality results on the FHK conjectures, or other extreme and typical statistics in random matrix theory.

\subsection{Notations}

In this paper we denote $c, C > 0$ generic small and large constants which do not depend on $N$ but may vary from line to line. Also let
\begin{equation}
		\vphi = e^{C_0 (\log\log N)^2}
\end{equation}
be a subpolynomial error parameter, for some fixed $C_0 > 0$. The constant $C_0$ is chosen large enough so that the local law and rigidity estimates,  Theorems~\ref{SLSL} and~\ref{rigidity}, hold.

\begin{definition}
		A generalized Wigner matrix $H = H(N)$ is a Hermitian or symmetric $N \x N$ matrix whose upper-triangular elements $H_{ij} = \conj{H_{ji}}$, $i \leq j$, are independent random variables with mean zero and variances $\si_{ij}^2 = \E \l[ |H_{ij}|^2 \r]$ that satisfy the following two conditions:
		\begin{itemize}[-]
				\item Normalization: for any $1 \leq j \leq N$, $\sum_{i=1}^N \si_{ij}^2 = 1$.
				\item Non-degeneracy: $\si_{ij}^2 \asymp  N^{-1}$ uniformly in $N$ and  $1 \leq i, j \leq N$.
		\end{itemize}
		In the Hermitian case we also assume $\Var \Re(H_{ij}) \asymp \Var \Im(H_{ij})$,  also uniformly in the parameters,  and independence of $\Re(H_{ij})$ and $\Im(H_{ij})$.
\end{definition}
We also assume 
an exponential decay estimate on the matrix entries: there exists $c > 0$ such that for all $i, j, N$ and $x > 0$ we have
\begin{equation}\label{a:exp-decay}
		\P \l( | \sqrt{N} H_{ij} | > x	\r) \leq c^{-1} e^{- x^c}.
\end{equation}

We will always order and denote the eigenvalues $\la_1 \leq \cdots \leq \la_N$ of $H$, and $\mu_1 \leq \cdots \leq \mu_N$ for those of a GOE or GUE matrix. We denote the limiting spectral density for Wigner matrices as
\[
		\rsc(x) = \frac{1}{2\pi} \sqrt{(4 - x^2)_+}.
\]
In this paper we consider the Kolmogorov distance,
\begin{equation}\label{eqn:Kolmogorov}
		d_{\t{K}}(X, Y) = \sup_x \l|  \P(X \leq x) - \P(Y \leq y) \r|,
\end{equation}
as our way of measuring the rate of convergence.

Our main results in Section \ref{sec:main} hold for possibly atomic distributions for the matrix entries, provided they are supported on enough points,   uniformly  in the following sense.

\begin{definition}\label{def:psup}
We say that a probability measure $\mu$ has $p$-support  if there exists $c>\tilde c,\kappa,A$ and  $p$ points  $x_1,\dots,x_p$
such that $\min_{i\neq j}|x_i-x_j|>c$,  $\mu[x_i,x_i+\tilde c]>\kappa$, $\max_i|x_i|\leq A$.

We say that a collection $(\mu_\alpha)_{\alpha\in I}$ of probability measures has uniform $p$-support if each $\mu_\alpha$ has $p$-support with $c>\tilde c,\kappa,A$ being uniform in $\alpha$.

A sequence of symmetric random matrices $(H_N)$ is also said to have  uniform $p$-support is the collection of distributions of $(\sqrt{N}H_{ij})_{N,1\leq i\leq j\leq N}$ has   uniform $p$-support.
In the Hermitian case,  the definition states that $(\sqrt{N}{\rm Re}H_{ij})_{N,1\leq i\leq j\leq N}$ and 
$(\sqrt{N}{\rm Im}H_{ij})_{N,1\leq i\leq j\leq N}$
 have  uniform $p$-support.

In particular, for both universality classes,  if the $\sqrt{N}H_{i\leq j}$ are i.i.d.  with $p$-support,  $(H_N)$ has uniform $p$-support.
\end{definition}

A much more restrictive assumption is smoothness of the matrix entries. A sequence $(H^{(N)})_N$ of generalized Wigner matrices is said to be smooth on scale $\si = \si(N)$ if $\sqrt{N} H_{ij}$ has density $e^{-V}$, where the potential $V = V_{N, i, j}$ satisfies the following condition uniformly in $N, i, j$: For any $k \geq 0$ there exists a $C_k > 0$ such that
\begin{equation}\label{c:smooth}
		|V^{(k)}(x)| \leq C_k \si^{-k} (1 + |x|)^{C_k}, \quad x \in \R.
\end{equation}
Under this smoothness condition, Bourgade \cite{Bo22} obtained quantitative rates of convergence on the extreme gaps process of Wigner matrices. Our result, Theorem~\ref{t:min-gap-rate} does not require this smoothness condition since we do not rely on the reverse heat flow method (see e.g. \cite{Bo22}*{Lemma 4.1}) and instead extend the Green function comparison machinery to higher moments.

\subsection{Main results} \label{sec:main}

We now state our main results on quantitative universality for gaps of Wigner matrices. Note that all of our results concern the bulk of the spectrum. The analogous best known rates in the edge at the time of writing are given by \cite{ScXu23}.

\begin{theorem}[Convergence rate for a single gap]\label{t:gap-rate}
		Let $\a, \varepsilon > 0$ be fixed. 
		Let $\mu_1 \leq \cdots \leq \mu_N$ be the eigenvalues of a GOE matrix. Then there exists an $p= p(\varepsilon)$  such that if $\lambda_1 \leq \cdots \leq \lambda_N$ are the eigenvalues of  a real Wigner ensemble with  uniform $p$-support,   for all large enough $N$, uniformly in $\a N \leq k \leq (1 - \a) N$, we have
		\begin{equation}\label{e:dK-single-gap}
				d_{\t K} \l( N (\la_{k+1} - \la_k ), N(\mu_{k+1} - \mu_k) \r) \leq N^{- \frac{1}{2} + \varepsilon}.
		\end{equation}
		The analogue result holds for the Hermitian symmetry class.
\end{theorem}

The above bound $N^{ - \frac{1}{2} + \varepsilon}$ gives a very explicit, reasonable convergence speed, but in view of (i) the
error estimate $(N^2t)^{-1}$ in \cite{Bo22}*{Corollary 3.3} (ii) the Berry-Esseen theorem and the $N^2$ independent sources of randomness for Wigner matrices,  it seems plausible that the optimal error term  is $N^{-1}$, as stated below.

\begin{conjecture}[Optimal speed of convergence]\label{optimal}
		Let $\a, \varepsilon > 0$ be fixed. Then there exists an $p = p(\varepsilon)$ sufficiently large such that if $(H_N)$ is a generalized Wigner ensemble with uniform $p$-support,   for all large enough $N$, uniformly in $\a N \leq k \leq (1 - \a) N$, we have
$
				d_{\t K} \l( N (\la_{k+1} - \la_k ), N(\mu_{k+1} - \mu_k) \r) \leq N^{- 1 + \varepsilon}.
$
Moreover,  a stronger conjecture states that $p=2$ is enough. 
\end{conjecture}

We also have the associated result for the smallest gap of a \emph{complex} Wigner matrix: 
\begin{theorem}[Convergence for the smallest gap, complex case]\label{t:min-gap-rate}
		Let $\a> 0$ be fixed. Let $\mu_1 \leq \cdots \leq \mu_N$ be the eigenvalues of a GUE matrix. Then there exists an $p$  such that if $\lambda_1 \leq \cdots \leq \lambda_N$ are the eigenvalues of  a complex Wigner ensemble with  uniform $p$-support,  for all large enough $N$ we have 
		\begin{equation}\label{eqn:compMin}
				d_{\t K} \l(\min_{\a N \leq k \leq (1 - \a) N} N^{ \frac{4}{3} } (\la_{k+1} - \la_k), \min_{\a N \leq k \leq (1 - \a) N} N^{ \frac{4}{3} } ( \mu_{k+1} - \mu_k) \r)={\rm o}(1).
		\end{equation}
In particular combining the above theorem with \cite{BABo13}*{Corollary 1.5},  there exists a constant $c=c(\alpha)$ such that for any $x\geq 0$,
\[
\mathbb{P}\left(c\min_{\a N \leq k \leq (1 - \a) N} N^{ \frac{4}{3} } (\la_{k+1} - \la_k)\leq x\right)\to \int_0^x 3u^2 e^{-u^3}{\rm d} u.
\]
\end{theorem}

\begin{remark}
The required $p$ above can be made explicit,  e.g.    the choice $p \geq 10$ (say) suffices, which may be seen through our proof.
\end{remark}

\begin{remark}
In (\ref{eqn:compMin}) one can consider $\min_{\a N \leq k \leq (1 - \a) N} c_k N^{ \frac{4}{3} } (\la_{k+1} - \la_k)$ (and similarly with $\pmb{\mu}$) with $c_k\asymp 1$ instead,  and the same result holds.
\end{remark}

For the real symmetric class, we have the following conjecture, which similarly states convergence of the smallest gap to the smallest gap for GOE.  
\begin{conjecture}[Convergence for the smallest gap,  real case]\label{t:min-gap-rate-real}
		Let $\a > 0$ be fixed. Let $\mu_1 \leq \cdots \leq \mu_N$ be the eigenvalues of a GOE matrix. Then there exists $p\geq 2$ sufficiently large such that if $(H_N)$ is a real random matrix from the generalized Wigner ensemble,  with  uniform $p$-support,  for all large enough $N$ we have 
		\[
				d_{\t K} \l(\min_{\a N \leq k \leq (1 - \a) N} N^{ \frac{3}{2} } (\la_{k+1} - \la_k), \min_{\a N \leq k \leq (1 - \a) N} N^{ \frac{3}{2} } ( \mu_{k+1} - \mu_k) \r) ={\rm o}(1).
		\]
A stronger conjecture states that $p=2$ is enough,  both for the the result above and  for (\ref{eqn:compMin}).

In particular, combined with \cite{FeTiWe19}*{Corollary 1} we expect that for some $C=C(\alpha)$ we have, for any $x\geq 0$, 
\[
\mathbb{P}\left(C\min_{\a N \leq k \leq (1 - \a) N} N^{ \frac{3}{2} } (\la_{k+1} - \la_k)\leq x\right)\to \int_0^x 2u e^{-u^2}{\rm d} u.
\]
\end{conjecture}

Although this conjecture does not seem accessible with the methods developed in this paper,  the following weaker statement is within reach e.g. see the remark after Theorem~\ref{4-moment-thmbis-min}. 
It corresponds to any scale greater than the smallest gap in the case of the real symmetry class:
		For any fixed $\a, \d> 0$ there exists a $p = p(\varepsilon)$ sufficiently large such that if $(H_N)$ is a real generalized Wigner matrix with uniform $p$-support,   for all large enough $N$ we have
		\begin{equation}\label{HN-GOE-min}
				 \P^{H_N} ( \min_{\a N \leq k \leq (1 - \a) N} N^{ \frac{3}{2}  } (\la_{k+1} - \la_k) \geq N^\d ) \leq N^{ - \frac{1}{2} + \varepsilon}.
		\end{equation}

The reason for the small $\d > 0$ discrepancy between the complex and real cases is that our Green function comparison estimate  will be of type
\begin{equation}\label{eqn:MomentMatch}
		\E \l[ \Tr \l( \frac{1}{H_N - z} \r) - \Tr \l(  \frac{1}{H^{\t{GOE}} - z} \r) \r] \leq N^q \frac{N^{- p/2}}{(N\h)^{p}}
\end{equation}
for some fixed constant $q > 0$, with $z = E + \i\h$. In particular, only when $\h \gg N^{- 3/2}$ can we take $p$ fixed,  sufficiently large to identify the Green function on scale $\eta$.

\begin{remark}
There is no obvious guess for the optimal error term in Theorem \ref{t:min-gap-rate} and Conjecture \ref{t:min-gap-rate-real}, and it may depend on $p$ from the $p$-support assumption.

Regarding the error term  in  (\ref{HN-GOE-min}),  we expect that
the  probability of the smallest gap of a GOE matrix to live at scale larger than $N^{- \frac{3}{2} + \d}$ should be exponentially small in $N$,  and similarly for any generalized Wigner matrix. 
\end{remark}

\section{Rate of convergence for gaps between eigenvalues}

As is standard in the three-step approach to universality (see e.g.  \cite{ErYa15}), we aim to combine a moment matching estimate with relaxation via dynamics. In particular, we let $\pmb{\la}(t)$ and $\pmb{\mu}(t)$ be the eigenvalues of two copies of the matrix Ornstein-Uhlenbeck process
\begin{equation}\label{e:DBM}
		{\rm d} H_t = \frac{1}{\sqrt{N}}{\rm d}B_t - \frac{1}{2}H_t {\rm d}t,
\end{equation}
where the initial conditions $\pmb{\la}(0)$ and $\pmb{\mu}(0)$ are given by an independent generalized Wigner matrix and GOE, respectively. From \cite{Bo22}*{Corollary 3.3}, the following relaxation estimate regarding an individual eigenvalue gap is known:
\begin{theorem} There exists a coupling between the matrix evolutions under (\ref{e:DBM}) such that the following holds.
		Let $\a, c> 0$ be fixed, arbitrarily small. Then for any (large) $D > 0$, there exist $C, N_0$ such that for any $N > N_0$, $\vphi^C / N < t < 1$ and $\a N \leq k \leq (1 - \a) N$ we have
		\begin{equation}\label{e:gaps-relaxation}
				\P \l(  |(\la_{k+1}(t) - \la_k(t)) - (\mu_{k+1}(t) - \mu_k(t))| > \frac{N^c}{N^2 t} \r)\leq N^{-D}.
		\end{equation}
\end{theorem}

The relaxation estimate (\ref{e:gaps-relaxation}) is a component enabling the three-step strategy in our paper. Our contribution is in the moment-matching component. Altogether, our approach to universality for a generalized Wigner matrix $H$ may be summed up in Figure~\ref{fig:1}, where we have labeled the current best error estimates for the two respective components.

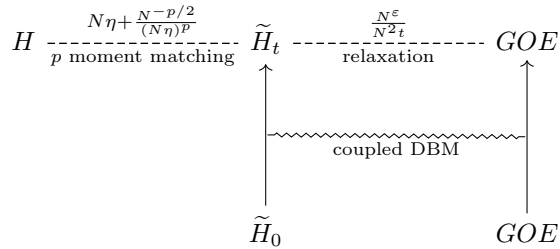
\begin{figure}[H]
		\centering
		\begin{tikzcd}
			H &&& {\tilde H_t} &&& GOE \\
			\\
			\\
			&&& {\tilde H_0} &&& GOE
			\arrow["{N\eta + \frac{N^{-p/2}}{(N\eta)^p} }", dashed, no head, from=1-1, to=1-4]
			\arrow["{p \text{ moment matching}}"', dashed, no head, from=1-1, to=1-4]
			\arrow["{\frac{N^\varepsilon}{N^2 t}}", dashed, no head, from=1-4, to=1-7]
			\arrow["{\text{relaxation}}"', dashed, no head, from=1-4, to=1-7]
			\arrow[""{name=0, anchor=center, inner sep=0}, from=4-4, to=1-4]
			\arrow[""{name=1, anchor=center, inner sep=0}, from=4-7, to=1-7]
			\arrow["{\text{coupled DBM}}", squiggly, no head, from=1, to=0]
		\end{tikzcd}
\caption{Three-step approach to universality with $p$ moments matching.  The total error for statistics on scale $\eta$ is $N\eta + \frac{N^{-p/2}}{(N\eta)^p}+\frac{N\varepsilon}{N^2 t}$, $t\ll 1$.}
\label{fig:1}
\end{figure}

What is  missing is the existence of a GDE matrix $\tilde H_t$, matching the first $p$ moments with a given generalized Wigner matrix $H$.  For this we proceed as follows.

We first state the following lemma as in \cite{ErYaYi11}*{Lemma 3.4} (which we will not use) for sake of comparison with our existence result, Lemma~\ref{GDE-moment-matching}.  The result below covers the Bernoulli case but only matches the first three moments and the fourth with an $O(t)$ error.
\begin{lemma}\label{lem:match3}
		Let $m_3$ and $m_4$ be two real numbers such that
		\[
				m_4 - m_3^2 - 1 \geq 0, \quad m_4 \leq C_2
		\]
		for some positive constant $C_2$. Let $U$ be a real Gaussian random variable with mean 0 and variance 1. Then for any sufficiently small $t > 0$ (depending on $C_2$), there exits an independent real random variable $h_0$, centered with variance 1 and with subexponential decay,  for which the first four moments of
		\[
				h_t := (1 - t)^{1/2} h_0 + t^{1/2} U
		\]
		are $m_1(h_t) = 0$, $m_2(h_t) = 1$, $m_3(h_t) = m_3$, and $m_4(h_t)$ such that
		\[
				|m_r(h_t) - m_4| = C t,
		\]
		where $C = C(C_2) > 0$.
\end{lemma}
In the higher moment matching below,  for any probability measure $\mu$ on $\mathbb{R}$ we denote 
$m_k(\mu)=\int x^k{\rm d}\mu(x)$ and $\pmb{m}_p(\mu) = (m_1(\mu), \ldots, m_p(\mu)) \in \R^p$.

\begin{lemma}\label{l:int-moment-simplex}
		Fix a $p \geq 1$ and $\mu$ a probability measure with $p$-support in the sense of Definition \ref{def:psup}, with parameters $c,\tilde c,\kappa,A$. 
		Then there exists $\delta=\delta(c,\tilde c,\kappa,A,p)>0$ such that for any $\pmb m\in\mathbb{R}^p$ 
		such that $|\pmb m-\pmb{m}_p(\mu)|<\delta$ there exists $\nu$ coinciding with $\mu$ on $|x|>A+1$ such that $\pmb{m}_p(\nu)=\pmb{m}$.
\end{lemma}

\begin{remark}
By analogy with the case $p=2$ from Lemma \ref{lem:match3} one expects that it is even possible to match $p+1$ moments,  i.e.  for any $\pmb m\in\mathbb{R}^{p+1}$ close to $\mu_{p+1}(\mu)$ find $\nu$ close to $\mu$ such that $\pmb{m}_p(\nu)=\pmb{m}$.  We don't consider  such an improvement here: The dependence $p=p(\varepsilon)$ in our main results,  e.g.  in Theorem \ref{t:gap-rate},  is anyways suboptimal from other aspects of the proof.
\end{remark}

\begin{proof}
From Definition \ref{def:psup}, $\mu$ can be written as
$\mu=\sum_{k=1}^p\mu_k+\nu$ where $\mu_k$ is supported on $[x_k,x_k+\tilde c]$,  $\mu_k[x_k,x_k+\tilde c]>\kappa$ and $\min_{i\neq j}|x_i-x_j|>c$, and $\mu_1,\dots,\mu_p,\nu$ are positive measures. Without loss of generality we assume $x_1>\dots > x_p$.

For any $\pmb{\delta}=(\delta_1,\dots,\delta_p)\in\mathbb{R}^p$ we define the probability measure $\mu^{\pmb{\delta}}$ through
$\mu^{\pmb{\delta}}(B)=\sum_{k=1}^p\mu_k(B-\delta_k)+\nu(B)$ for any Borel set $B$. Then
$m_\ell(\mu^{\pmb{\delta}})=\sum_{k=1}^p\int(y+\delta_k)^\ell\mu_k({\rm d} y)+\int y^\ell\nu({\rm d}y)$, so that
$\partial_{\delta_k}m_\ell(\mu^{\pmb{\delta}})\mid_{\delta_k=0}=\int\ell y^{\ell-1}\mu_k({\rm d} y)$.
Therefore at $\pmb{\delta}=0$ the Jacobian matrix $J=(\partial_{\delta_k}m_\ell(\mu^{\pmb{\delta}}))_{k,\ell}$ satisfies
\begin{multline*}
\underset{p\times p}{{\rm det}}\ J=\underset{p\times p}{{\rm det}}\ \int \ell y^{\ell-1}\mu_k({\rm d} y)
=
\frac{1}{p!}\int_{\mathbb{R}^p} \underset{p\times p}{{\rm det}} (\ell y_i^\ell)\ \cdot\  \underset{p\times p}{{\rm det}} \mu_k({\rm d} y_i)
=
\int_{\mathbb{R}^p} \prod_{1\leq i<j\leq p}(y_i-y_j)\cdot\  \underset{p\times p}{{\rm det}} \mu_k({\rm d} y_i)\\
=p!\int_{y_1>\dots>y_p} \prod_{1\leq i<j\leq p}(y_i-y_j)\cdot\  \underset{p\times p}{{\rm det}} \mu_k({\rm d} y_i),
\end{multline*}
where we have used Andr\'eief's identity (see e.g.  \cite{BouDub2020}*{Lemma 5.1}) in the second equality,  the Vandermonde determinant in the third one, and invariance of the product of determinants under permutation of the indices in the fourth one. As the supports of the $\mu_k$'s are disjoint in decreasing order,  in the domain $y_1>\dots >y_p$ we have ${{\rm det}} \mu_k({\rm d} y_i)=\mu_1({\rm d}y_1)\dots \mu_p({\rm d}y_p)$, so we have proved
\begin{equation}\label{eqn:det}
\underset{p\times p}{{\rm det}}\ J\geq p! (c-\tilde c)^{\frac{p(p-1)}{2}} \kappa^p.
\end{equation}
In particular $\pmb{\delta}\mapsto \pmb{m}_p(\mu^{\pmb{\delta}})$ is locally invertible around $\pmb{\delta}=0$,  so  
for $\pmb m$ close enough to $\pmb{m}_p(\mu)$ there exists a $\nu=\mu^{\pmb{\delta}}$ such that  $\pmb{m}_p(\nu)= \pmb{m}$. By construction $\mu=\nu$ on $|x|>A+1$.

Regarding uniformity of $\delta$ with respect to $c,\tilde c,\kappa,A$,  the result immediately follows from the observations that $\det(JJ^*)>c_0(c,\tilde c,\kappa)$ from (\ref{eqn:det}),   $\|JJ^*\|_{2\to2}={\rm O}_p(A^{2p})$  and $\max_{i,j,\ell,|\pmb{\delta}|<1}|\partial_{\delta_i,\delta_j}m_\ell(\mu^{\pmb{\delta}})|={\rm O}_p(A^p)$.
\end{proof}

\noindent In the following lemma, we write $m_k(X)$ for $m_k(\mu)$ when $X$ is a real random variable with distribution $\mu$.

\begin{lemma}\label{GDE-moment-matching}
	Fix a $p \geq 2$ 
	and assume that the distribution of a probability measure $\mu$ has $p$-support in the sense of Definition \ref{def:psup}
	with parameters $c,\tilde c,\kappa,A$.  Assume $m_1(\mu)=0$, $m_2(\mu)=1$.  Then there exists $t_0 = t_0(c,\tilde c,\kappa,A,\pmb{m}_p(\mu)) > 0$ such that for any $0\leq t\leq t_0$, 
	there exists a probability measure $\nu$
coinciding with $\mu$ on $\{|x|>A+1\}$ such that if $X,G$ are independent, $X\sim \nu$, $G$ is a standard Gaussian,  then
		\begin{equation}\label{GDE-form}
				 X_t := e^{-t/2}  X + (1 - e^{-t})^{1/2} G,
		\end{equation}
has its first $p$ moments match exactly those of $\mu$. That is, for $j \leq p$ we have
		\begin{equation}\label{moment-constraint}
				m_j( X_t) = m_j(\mu).
		\end{equation}
\end{lemma}

\begin{proof}
		By a change of variables, without loss of generality we can  look for $ X$ such that
$
				 X_t = (1 - t)^{1/2}  X + t^{1/2} G
$
has its first $p$ moments match exactly those of $\mu$.
Denote $\pmb{m}_p=(m_1(X),\dots,m_p(X))$ any sequence of moments ensuring this moment matching.
		Then for the first two moments, we wish to satisfy
		\begin{align*}
				0 &= m_1( X_t) = (1 - t)^{1/2} m_1 ( X) \\
				1 &= m_2( X_t) = (1-t) m_2 ( X) + t.
		\end{align*}
		So necessarily $m_1( X) = 0, m_2( X) = 1$.  Now suppose there is a choice of $m_j( X)$ that makes (\ref{moment-constraint}) hold for all $j \leq k-1$.  We show that there is also a unique choice of $m_{k}( X)$ which additionally ensures $m_{k}( X_t) = m_k$.   Necessarily
		\begin{align}
				m_k(\mu) = m_k( X_t) &= \sum_{j=0}^k \binom{k}{j} (1-t)^{ \frac{k-j}{2} } t^{ \frac{j}{2} } m_{k-j}( X) m_{j}(G) 
		\end{align}
		where we set $m_0( X) = m_0(G) = 1$.  This gives
		\begin{equation}\label{k-constraint}
				m_k( X) = \frac{m_k(\mu) - \sum_{j=1}^k \binom{k}{j} (1 - t)^{ \frac{k-j}{2}	} t^{ \frac{j}{2} } m_{k-j}( X) m_{j}(G)}{(1-t)^{k/2}}.
		\end{equation}
		Consider $\delta$ from Lemma \ref{l:int-moment-simplex}.  By continuity in $t$ of the above construction of $\pmb{m}_p$, there exists a $t_0=t_0(\delta,\pmb{m}_p(\mu))$ such that for any $t<t_0$ we have 
		$|\pmb{m}_p-\pmb{m}_p(\mu)|<\delta$.   Lemma \ref{l:int-moment-simplex} implies the existence of a $\nu$ as desired such that $\pmb{m}_p=\pmb{m}_p(\nu)$, concluding the proof.
\end{proof}

\subsection{Green function comparison via moment matching}
We use a superscript to denote the underlying matrix distribution of the object in question. In what follows, we use $\bf{v}/\sqrt{N},\bf{w}/\sqrt{N}\in\mathbb{R}^{\frac{N(N+1)}{2}}$ to denote the entrywise distributions of two generalized Wigner ensembles $H^{\bf{v}}$ and $H^{\bf{w}}$, respectively, which we assume to be in the real symmetry class: The proof for the Hermitian symmetry class is the same up to notational changes.

In comparing the resolvents between the $\bf{v}$ and $\bf{w}$ ensembles, we proceed by a telescopic expansion driven by a Lindeberg exchange principle, as initiated for random matrix universality in \cite{TaVu11}. 
As is customary with this swapping method in random matrix theory,  we fix a bijective ordering map of the index set of our matrix, $\Phi: \{(i,j): 1 \leq i \leq j \leq N\} \to \{1, \ldots, \ga(N)\}$ where $\ga(N) = N(N+1)/2$. Denote by $H_\ga$ the generalized Wigner matrix whose entries $H_{ij}$ follow the $\tb{w}/\sqrt{N}$-distribution if $\Phi(i,j) \leq \ga$ and the $\tb{v}/\sqrt{N}$-distribution otherwise. Then $H_0 = H^{\tb{v}}$ and $H_{\ga(N)} = H^{\tb{w}}$. We will denote 
$\mathbb{P}_\gamma,\mathbb{E}_\gamma$ when integrating over the ensemble $H_\gamma$.

For the following notation, fix a $\ga = \Phi(a,b)$, and let $D^{(ij)}$ denote the matrix that is zero everywhere except at position $(i,j)$ and $(j,i)$ where it is 1. Then we have
\begin{equation}
		H_{\ga - 1} = Q + \frac{v_{ab}}{\sqrt{N}}  D^{(ab)}, \ 
		H_\ga = Q + \frac{w_{ab}}{\sqrt{N}}  D^{(ab)}\label{eqn:gamma0}.
\end{equation}

To move between the Green function and the gap probabilities, we introduce some notation regarding the Green function approximation.    Let $\rho > 0$ be fixed and $\rho N \leq k \leq (1- \rho)N$ be a bulk index, and define the rigidity window
\begin{equation}\label{Ia}
		I_k := \l[ \ga_k - N^{-1} \vphi^C, \ga_k + N^{-1} \vphi^C \r],
\end{equation}
where the $k$-th typical location $\ga_k$ is defined implicitly through $\int_{-\infty}^{\ga_k} d\rsc = \frac{k}{N}$. 
We now 
fix 
\[
\varepsilon > 0, \ \  \eta_d = N^{-3/2+\varepsilon},\ \ \eta_0=N^{-3/2+2 \varepsilon},\ \ \Delta=N^{-3/2+3\varepsilon}.\]
  For the proof of the theorem  below,  we also set 
\[
s(x)=\frac{1}{1+e^{x}}
\]
and for any $E \in [-3, 3]$
\begin{equation}\label{eqn:fE}
f_E(x)=s\left(\frac{x-E}{\eta_d}\right).
\end{equation}
Then $f_E(x)$ is a good approximation of $\1_{x\leq E}$ on scale $\eta_d$,  see Lemma \ref{smoothing-approx}.
Note that $f_E$ is an analytic function with simple poles at $z_k=E+{\rm i}\eta_d(2k+1)$,  $k\in\mathbb{Z}$.
Moreover,  we will use the abbreviation $\Tr_E = \Tr f_E(H)$. 

Finally,  let $q_\a$ be a smooth, monotone cutoff function satisfying
\begin{equation}\label{qa}
		q_\a(x) = \begin{cases}
				1 & x \geq \a - 1/3 \\
				0 & x \leq \a - 2/3	
		\end{cases}.
\end{equation}

\begin{theorem}\label{4-moment-thmbis}
		Let $\varepsilon, \a > 0$ be fixed. Suppose that the $\bf{v}$ and $\bf{w}$ matrix entries have matching moments up to the $p$-th order, and that $\bf{w}$ is a Gaussian-divisible distribution as in (\ref{GDE-form}) with $t\asymp 1$.  Then there exists $C>0$ such that   uniformly in $\a N \leq k \leq (1 - \a) N$ and $0<u<N^{\varepsilon}$, we have
		\begin{equation}\label{4-moment-estbis}
				\left|(\P^{\tb{v}} - \P^{\tb{w}})  (N (\la_{k+1} - \la_{k})\geq u) \right|\leq \varphi^C N^2 \frac{N^{-(p+1)/2}}{(N\h_d)^{p+1}}+N^{1+6\varepsilon}\eta_d,\ \ \eta_d=N^{-3/2+\varepsilon}.
		\end{equation}
\end{theorem}

\begin{proof}
		Fix $\a N \leq k \leq (1 - \a)N$ and let $\b_j = \ga_{k} + j\D$ for $j_- \leq j \leq j_+$ where $j_\pm = \pm \lfloor N^{1/2 + \varepsilon} \rfloor$. We start by decomposing and applying (\ref{smoothing-est}): together with the rigidity estimate (\ref{eqn:rigidity}) we have
		\begin{align*}
				\P^{\bf{v}}(N(\la_{k+1} - \la_k) \geq u) &\leq \sum_{j_- \leq j \leq j_+} \P^{\bf{v}}\l( \la_k \in [\b_j, \b_{j+1}), \la_{k+1} \in [\b_j + \frac{u}{N}, \infty)  \r)  +{\rm O}(N^{-D})\\
				&\leq \sum_{j_- \leq j \leq j_+} \E^{\bf{v}} \l[ \1_{\la_k \in [\b_j, \b_{j+1})}\l( 1 -  q_{k+1} (\Tr_{\b_j + \frac{u}{N} - \h_0} \r)  \r]+{\rm O}(N^{-D}).
		\end{align*}
		Integrating the above sum by parts, we have
		\begin{multline*}
				\P^{\bf{v}}(N(\la_{k+1} - \la_k) \geq u) \leq \sum_{j_- + 1 \leq j \leq j_+} \E^{\bf{v}} \l[ \1_{\la_k \geq \b_j} \l( q_{k+1}(\Tr_{\b_{j-1} + \frac{u}{N} - \h_0}) - q_{k+1}(\Tr_{\b_{j} + \frac{u}{N} - \h_0}) \r) \r]+{\rm O}(N^{-D})\\
				+ \E^{\bf{v}} \l[ \1_{\la_k \geq \b_{j_-}}\l( 1 -  q_{k+1}(\Tr_{\b_{j_-} + \frac{u}{N} - \h_0})\r) \r] - \E^{\bf{v}} \l[ \1_{\la_k \geq \b_{j_+}} \l( 1 - q_{k+1}(\Tr_{\b_{j_+} + \frac{u}{N} - \h_0}) \r) \r]+{\rm O}(N^{-D}).
		\end{multline*}
		The boundary terms simplify by rigidity, and we apply (\ref{smoothing-est}) to see that (note that the difference in the $q_{k+1}$ terms is negative)
		\begin{multline*}
				\P^{\bf{v}}(N(\la_{k+1} - \la_k) \geq u) \\
				\leq 1 + \sum_{j_- + 1 \leq j \leq j_+} \E^{\bf{v}} \l[ \1_{\la_k \geq \b_j} \l( q_{k+1}(\Tr_{\b_{j-1} + \frac{u}{N} - \h_0}) - q_{k+1}(\Tr_{\b_{j} + \frac{u}{N} - \h_0}) \r) \r] + O(N^{-D}) \\
				\leq 1 + \sum_{j_- + 1 \leq j \leq j_+} \E^{\bf{v}} \l[ \l( 1 - q_k(\Tr_{\b_j + \h_0}) \r) \l( q_{k+1}(\Tr_{\b_{j-1} + \frac{u}{N} - \h_0}) - q_{k+1}(\Tr_{\b_{j} + \frac{u}{N} - \h_0}) \r) \r] + O(N^{-D}) .
		\end{multline*}
We now claim that due to our moment matching hypothesis,  for any entry swap at site $\gamma$ we have
almost distributional invariance in the following sense:
		\begin{multline}\label{eqn:swapp}
				\E_{\gamma-1} \l[ \l( 1 - q_k(\Tr_{\b_j + \h_0}) \r) \l( q_{k+1}(\Tr_{\b_{j-1} + \frac{u}{N} - \h_0}) - q_{k+1}(\Tr_{\b_{j} + \frac{u}{N} - \h_0}) \r) \r] \\
				= \E_{\gamma} \l[ \l( 1 - q_k(\Tr_{\b_j + \h_0}) \r) \l( q_{k+1}(\Tr_{\b_{j-1} + \frac{u}{N} - \h_0}) - q_{k+1}(\Tr_{\b_{j} + \frac{u}{N} - \h_0}) \r) \r] + {\rm O}\left( \varphi^C \frac{N^{-(p+1)/2}}{(N\h_d)^{p+1}}\right).
		\end{multline}
		To prove the above equation,  we consider (we abbreviate $q=q_k$,  $D=D^{ab}$)
		\[F(x)= q_{k}(\Tr_{E}(Q +x D))=q\circ h(x).\]
		We will prove that uniformly in $a,b, k$ and $E$ in the bulk of the spectrum
		we have
		\begin{equation}\label{eqn:invar}
		\mathbb{E}[F(\frac{v_{ab}}{\sqrt{N}})]=
		\mathbb{E}[F(\frac{w_{ab}}{\sqrt{N}})]+{\rm O}\left( \varphi^C \frac{N^{-(p+1)/2}}{(N\h_d)^{p+1}}\right).
		\end{equation}
		The case of the product of two functions of type $q_k$ can be proved similarly, concluding the proof of (\ref{eqn:swapp}).

To prove (\ref{eqn:invar}),  abbreviating $v=v_{ab}, w=w_{ab}$, we have
		\begin{equation}\label{eqn:F}
		\E F(\frac{v}{\sqrt{N}})=\sum_{\ell=0}^p	\frac{1}{\ell!}\E[F^{(\ell)}(0)]\E[\frac{v^k}{N^{k/2}}]+{\rm O}_{p}\left(\frac{\E[\|F^{(p+1)}\|_\infty]}{N^{(p+1)/2}}\right),
		\end{equation}
		where the notation for the norm used above and in the remainder of this paper is: for any function $g$ we denote \[
		\|g\|_{\infty}=\sup_{|x|<N^{-\frac{1}{2}+\kappa}}|g(x)|
		\] with $0<\kappa<1/2$ fixed, arbitrarily small ($\kappa=\varepsilon/100$ for example works).  Therefore,  from the moment matching hypothesis, and the exponential decay assumption (\ref{a:exp-decay}), the proof of
	 (\ref{eqn:invar}) will be complete if we have the key stability estimate
		\begin{equation}\label{eqn:Stab}
		\|F^{(p+1)}\|_{\infty}\leq C_p \frac{\varphi^C}{(N\eta_d)^{p+1}}
		\end{equation}
		with overwhelming probability.  As $q$ is smooth on scale $1$, by the chain rule to prove (\ref{eqn:Stab}) we only need to prove that for any $\varepsilon>0$ and $\ell\leq p$,  with overwhelming probability
		\begin{equation}\label{eqn:derBound}
		\|h^{(\ell)}\|_{\infty}\leq C_\ell \frac{\varphi^C}{(N\eta_d)^\ell},
		\end{equation}
First note that with overwhelming probability,  for any $z=E+{\rm i}\eta$ with $0<\eta\lesssim 1$  we have 
		\begin{equation}\label{eqn:est}
		\left\|\partial_x^\ell{\rm Tr}\frac{1}{Q-z+x D}\right\|_{\infty}\leq  \frac{C_\ell}{\eta} \frac{\varphi^C}{(\min(1,N\eta))^\ell}.
		\end{equation}
		We start the proof of (\ref{eqn:est}) in the most singular case $0<\eta<1/N$.
		From the resolvent expansion we have
				\begin{multline}\label{eqn:expansion}
				\frac{1}{Q-z+(x+\delta) D} \\= \frac{1}{Q-z+x D}  + \l[ \sum_{k=1}^{\ell} (- \frac{1}{Q-z+x D} \delta D)^k  \frac{1}{Q-z+x D}  \r] + (- \frac{1}{Q-z+x D} \delta D)^{\ell+1}  \frac{1}{Q-z+(x+\delta) D}
		\end{multline}
		so that
		\[
		\partial_x^\ell{\rm Tr}\frac{1}{Q-z+x D}=\ell!\ {\rm Tr} \left(- \frac{1}{Q-z+x D} D\right)^\ell  \frac{1}{Q-z+x D}.
		\]
		The estimate (\ref{eqn:est}) then follows from the above formula together with the local law (\ref{SLSL}) (see in particular (\ref{eqn:Extend})) and the Ward identity $\sum_j|G_{ij}|^2=\frac{1}{\eta}{{\rm Im} G_{ii}}$ where we abbreviate $G=(Q-z+xD)^{-1}$:
		\begin{multline*}
		\left|\partial_x^\ell{\rm Tr}\frac{1}{Q-z+x D}\right|\leq  \ell!|{\rm Tr}(GD)^{\ell-1} G|\leq \sum_{1\leq i\leq N,j,k\in\{a,b\}}\ell!|{\rm Tr}(GD)^{\ell-1} G|\\
		\leq \sum_{j,k\in\{a,b\}}
		|[D(GD)^{\ell-1}]_{jk}|\sum_{i=1}^N(|G_{ij}|^2+|G_{ki}|^2)\leq \frac{1}{\eta}
		\sum_{j,k\in\{a,b\}}
		|[D(GD)^{\ell-1}]_{jk}|(|G_{jj}|+|G_{kk}|)\leq \frac{C_\ell}{\eta}\frac{\varphi^{C}}{(N\eta)^\ell}.
		\end{multline*}
		
		The proof of (\ref{eqn:est}) in the case $1/N<\eta\lesssim 1$ is the same, except that we use the bound $G_{ij}={\rm O}(\varphi^C)$ instead of $G_{ij}={\rm O}(\varphi^C/(N\eta))$ as an input, for the resolvent entries in the above expansion.
		Moreover,  the above reasoning, with (\ref{eqn:far}) and (\ref{eqn:far2}) as an input. also shows that for any $c>0$, uniformly in ${\rm dist}(z,[-2,2])>c$ we have
				\begin{equation}\label{eqn:est2}
		\left\|\partial_x^\ell{\rm Tr}\frac{1}{Q-z+x D}\right\|_{\infty}\leq C(\ell) \varphi^{C}.
		\end{equation}
	
	Let $\cal C$ is a circle with radius $R\in[5,6]$ such that ${\rm dist}(\cal C,\{z_k\})<\eta_d/10$, ensuring that $f=f_E$ from (\ref{eqn:fE}) satisfies $|f|={\rm O}(1)$ on $\cal C$.	
By the residue formula for any $|\lambda|<4$ we have (denoting $\cal D$ the disk with boundary $\cal C$)
		\[
		f(\lambda)=\int_{\mathcal{C}}\frac{f(z)}{z-\lambda}\frac{{\rm d}z}{2{\rm i}\pi}-\sum_{z_k\in\mathcal D}\frac{{\rm Res}(f,z_k)}{z_k-\lambda}=
		\int_{\mathcal{C}}\frac{f(z)}{z-\lambda}\frac{{\rm d}z}{2{\rm i}\pi}+\eta_d\sum_{z_k\in\mathcal D}\frac{1}{z_k-\lambda}.
		\]
		The above formula and eigenvalue rigidity (see Theorem \ref{rigidity}) imply
		\[
		|\partial_x^\ell {\rm Tr}_{E}(Q+xD)|\lesssim \sup_{\cal C}\left|\partial_x^\ell{\rm Tr}\frac{1}{Q-z+x D}\right|+\sum_{z_k\in{\cal D}}\left|\partial_x^\ell{\rm Tr}\frac{1}{Q-z_k+x D}\right|.
		\]
		Together with (\ref{eqn:est}) to estimate the contribution from the $z_k$'s, and (\ref{eqn:est2}) for the contribution of $\cal C$,  we obtain
		\begin{multline*}
		|\partial_x^\ell {\rm Tr}_{E}(Q+xD)|\lesssim \sup_{\cal C}\left|\partial_x^\ell{\rm Tr}\frac{1}{Q-z+x D}\right|+\eta_d\sum_{z_k\in{\cal D}}\frac{\varphi^C}{(2k+1)\eta_d(N(2k+1)\eta_d)^\ell}
		+\eta_d\sum_{z_k\in{\cal D}}\frac{\varphi^C}{(2k+1)\eta_d}\\
		\lesssim C_\ell
		 \frac{\varphi^C}{(N\eta_d)^\ell},
		\end{multline*}
		concluding the proof of (\ref{eqn:derBound}). 
		
With (\ref{eqn:swapp})  proved,  we now swap the expectation in $\bf{v}$ to an expectation in $\bf{w}$, with an error $\vphi^C N^{2} \frac{N^{-(p+1)/2}}{(N\h_d)^{p+1}}$. To continue, we apply (\ref{smoothing-est}) to yield
		\begin{multline*}
				\P^{\bf{v}}(N(\la_{k+1} - \la_k) \geq u) \\
				\leq 1 + \sum_{j_- + 1 \leq j \leq j_+} \E^{\bf{w}} \l[ \l( 1 - q_k(\Tr_{\b_j + \h_0}) \r) \l( q_{k+1}(\Tr_{\b_{j-1} + \frac{u}{N} - \h_0}) - q_{k+1}(\Tr_{\b_{j} + \frac{u}{N} - \h_0}) \r) \r] \\
				\qquad\qquad\qquad\qquad\qquad\qquad\qquad\qquad + N^2 \frac{N^{-(p+1)/2}}{(N\h_d)^{p+1}} + O(N^{-D})  \\
				\leq 1 + \sum_{j_- + 1 \leq j \leq j_+} \E^{\bf{w}} \l[ \1_{\la_k \geq \b_j + 2\h_0} \l( q_{k+1}(\Tr_{\b_{j-1} + \frac{u}{N} - \h_0}) - q_{k+1}(\Tr_{\b_{j} + \frac{u}{N} - \h_0}) \r) \r] \\
				 + \varphi^C N^2 \frac{N^{-(p+1)/2}}{(N\h_d)^{p+1}} + O(N^{-D}).  
		\end{multline*}
		We may replace $\1_{\la_k \geq \b_j + 2\h_0}$ with $\1_{\la_k \geq \b_j}$ and estimate the error  as follows,
		\begin{multline*}
				\sum_{j_- + 1 \leq j \leq j_+} \E^{\bf{w}} \l[ \1_{\la_k \in [\b_j, \b_j + 2\h_0)} \l| q_{k+1}(\Tr_{\b_{j-1} + \frac{u}{N} - \h_0}) - q_{k+1}(\Tr_{\b_{j} + \frac{u}{N} - \h_0}) \r| \r] \\
				\leq \sum_{j_- + 1 \leq j \leq j_+} \E^{\bf{w}} \l[ \1_{\la_k \in [\b_j, \b_j + 2\h_0)} \1_{\la_{k+1} \in [\b_{j-1} + \frac{u}{N} - 2\h_0, \b_j + \frac{u}{N})} \r] \
				= \E^{\bf{w}} \l[ \sum_{j_- + 1 \leq j \leq j_+} \1_{A_j} \r] = \E^{\bf{w}} \l[ \1_{A} \r] 
		\end{multline*}
		where we denoted $A = \sqcup_{j_- + 1 \leq j \leq j_+} A_j$ as the union of disjoint events $A_j$. We may now input the relaxation estimate (\ref{e:gaps-relaxation}), denoting the event
		\begin{equation}\label{eqn:E}
				E := \l\{ |(\la_{k+1}(t) - \la_k(t)) - (\mu_{k+1}(t) - \mu_k(t))| \leq \frac{N^\varepsilon}{N^2 t} \r\},
		\end{equation}
		where $\la_k$  follows a Gaussian divisible $\bm w$-distribution at time $t \sim 1$ and $\mu_k$ follows that of a GOE spectrum.  In particular, write
\[
				\E^{\bf{w}} \l[ \1_A \r] \leq \E^{\bf{w}} \l[ \1_{A \cap E} \r] + O(N^{-D}) 
										 \leq \E^{\bf{w}} \l[ \sum_{j_- + 1 \leq \ell \leq j_+} \1_{A \cap E \cap B_\ell}  \r] + O(N^{-D})
\]
		where $B_\ell := \{\mu_k \in [\b_\ell, \b_{\ell + 1}) \}$. Note that for time $t \sim 1$, the intersection $A \cap E \cap B_\ell$ implies that $\mu_{k+1} \in [\b_{\ell - 1} - 4\h_0 + \frac{u}{N} - \frac{c}{N^2}, \b_{\ell+1} + \frac{u}{N} + \frac{c}{N^2}) =: J_{\ell}$ for some constant $c > 0$. Therefore the last expectation boils down to a two-point correlation function estimate for the GOE: from \cite{Meh}*{Chapter 7} we have 
		$\rho^N_2(x,y)\lesssim N^{2}$ uniformly in $x,y$ in the bulk of the spectrum, so that $\mathbb{P}^{\t{GOE}}(B_\ell \cap J_\ell)\lesssim N^2\Delta^2$ and 
		\begin{equation}\label{e:wegner}
				\E^{\bf{w}} \l[ \sum_{j_- + 1 \leq \ell \leq j_+} \1_{A \cap E \cap B_\ell} \r] \leq \sum_{j_- + 1 \leq \ell \leq j_+} \E^{\t{GOE}} \l[ \1_{B_\ell \cap J_\ell} \r] 					\lesssim \sum_{j_- + 1 \leq \ell \leq j_+} (N\D)^2 \leq N^{-\frac{1}{2}+7\varepsilon}.
		\end{equation}
		So after replacing $\1_{\la_k \geq \b_j + 2\h_0}$ with $\1_{\la_k \geq \b_j}$, we integrate by parts again, apply rigidity, and (\ref{smoothing-est}) to get that
		\begin{multline*}
				\P^{\bf{v}}(N(\la_{k+1} - \la_k) \geq u) \\
				 \leq 1 + \sum_{j_- + 1 \leq j \leq j_+} \E^{\bf{w}} \l[ \1_{\la_k \geq \b_j} \l( q_{k+1}(\Tr_{\b_{j-1} + \frac{u}{N} - \h_0}) - q_{k+1}(\Tr_{\b_{j} + \frac{u}{N} - \h_0}) \r) \r]  + N^C \frac{N^{-(p+1)/2}}{(N\h_d)^{p+1}} + O(N^{- \frac{1}{2} + 7\varepsilon})  \\
				 \leq \sum_{j_- \leq j \leq j_+} \E^{\bf{w}} \l[ \1_{\la_k \in [\b_j, \b_{j+1})} \l(1 -  q_{k+1}(\Tr_{\b_{j} + \frac{u}{N} - \h_0}) \r) \r] +  N^C \frac{N^{-(p+1)/2}}{(N\h_d)^{p+1}} + O(N^{- \frac{1}{2} + 7\varepsilon}) \\
				 \leq \sum_{j_- \leq j \leq j_+} \E^{\bf{w}} \l[ \1_{\la_k \in [\b_j, \b_{j+1})}  \1_{\la_{k+1} \geq \b_j + \frac{u}{N} - 2\h_0} \r] +  N^C \frac{N^{-(p+1)/2}}{(N\h_d)^{p+1}} + O(N^{- \frac{1}{2} + 7\varepsilon}) .
		\end{multline*}
		Now, by the two point estimate as in  (\ref{e:wegner}),  and using that $E$ has overwhelming probability,  we have (remember $\Delta=N^{-3/2+3\varepsilon}$)
		\begin{multline*}
				\sum_{j_- \leq j \leq j_+} \E^{\bf{w}} \l[ \1_{\la_k \in [\b_j, \b_{j+1})}  \1_{\la_{k+1} \geq \b_j + \frac{u}{N} - 2\h_0} \r] \\
				\leq \P^{\bf{w}}(N(\la_{k+1} - \la_k) \geq u) + \sum_{j_- \leq j \leq j_+} \E^{\bf{w}} \l[ \1_{\la_k \in [\b_j, \b_{j+1})}  \1_{\la_{k+1} \in [\b_j + \frac{u}{N} - 2\h_0, \b_{j+1} + \frac{u}{N})} \r] \\
				\leq \P^{\bf{w}}(N(\la_{k+1} - \la_k) \geq u) + \sum_{j_- \leq j \leq j_+} O\l( (N\D)^2 \r) 
				\leq \P^{\bf{w}}(N(\la_{k+1} - \la_k) \geq u) + O \l( N^{- \frac{1}{2} + 7\varepsilon} \r).
		\end{multline*}
		Therefore we have shown
		\begin{equation}
				\P^{\bf{v}}(N(\la_{k+1} - \la_k) \geq u) \leq \P^{\bf{w}}(N(\la_{k+1} - \la_k) \geq u) + N^C \frac{N^{-(p+1)/2}}{(N\h_d)^{p+1}} + O(N^{- \frac{1}{2} + 7\varepsilon}).
		\end{equation}
		The proof of the reverse inequality proceeds similarly, starting with a lower bound on $\P^{\bf{v}}(N(\la_{k+1} - \la_k) \geq u)$ in terms of products of $q_\alpha({\rm Tr}_\beta)$,  then a moment matching to transfer to the ensemble $\bf{w}$,   and (\ref{eqn:E}) with 2-point estimates of GOE to adjust the size of the intervals by order $\eta_0$.
\end{proof}

\subsection{Proof of Theorem~\ref{t:gap-rate}}

If $H$ is a generalized Wigner matrix with uniform $p$-support for some $p$ to be chosen later, then Lemma~\ref{GDE-moment-matching} produces a matrix $\tilde H_0$ and a $t\asymp 1$ such that $\tilde H_t$ given by the matrix Ornstein Uhlenbeck process (\ref{e:DBM}) matches $p$ moments with $H$. Then via Theorem~\ref{4-moment-thmbis} and rigidity we have
\[
		d_{\t K} \l( N (\la_{k+1}^H - \la_k^H), N (\la_{k+1}^{\tilde H_t} - \la_k^{\tilde H_t}) \r) \leq \varphi^C N^2 \frac{N^{-(p+1)/2}}{(N\h_d)^{p+1}} + N^{1 +7 \varepsilon} \h_d.
\]
Combining this with the relaxation estimate (\ref{e:gaps-relaxation}) and the uniform (in $N$) bound on the density of $\mu_{k+1}-\mu_k$ (see  e.g. \cite{Meh}*{Chapter 7}) yields
\[
		d_{\t K} \l( N (\la_{k+1}^H - \la_k^H), N(\mu_{k+1} - \mu_k)  \r) \leq \varphi^C N^2 \frac{N^{-(p+1)/2}}{(N\h_d)^{p+1}} + N^{1 + 7\varepsilon} \h_d + \frac{N^{\varepsilon}}{N^2 t} + N^{- D}.
\]
Therefore,  with $\h_d = N^{- \frac{3}{2} + \varepsilon}$ and $p \geq \frac{10}{\varepsilon}$, say, we obtain (\ref{e:dK-single-gap}). \QED

\section{Universality and rate of convergence for the  smallest gap}%

\label{sec:Rate of convergence for smallest gap}

Before proving Theorem~\ref{t:min-gap-rate}, we need a concentration estimate which says that there cannot be too many eigenvalue gaps that live near the scale of the smallest gap.
\begin{lemma}\label{eqn:notmany}
		There exists a $p\geq 1$ such that for any  $\kappa,D > 0$ there exists a   $\d =\d(\kappa)> 0$ and $N_0 = N_0(\kappa)$ such that for $\bm{v}$ satisfying the uniform  $p$-support condition from Definition \ref{def:psup} and $N \geq N_0$ we have
		\begin{equation}\label{extreme-gap-stability}
				\P^{\bm{v}}\l(\# \l\{ \ga N \leq \a \leq (1 - \ga) N : \la_{\a+1} - \la_\a \leq N^{- \frac{4}{3} + \d}  \r\} > N^\kappa \r) < N^{-D}.
		\end{equation}
\end{lemma}

\begin{proof}
Recall the notation for $\h_d, \h_0, \D$, and $\b_j$ in the previous section. For $\ga N \leq \a \leq (1-\ga)N$, we can repeatedly apply (\ref{smoothing-est}) to bound
		\begin{align}
				&\1_{\la_{\a+1} - \la_\a \leq N^{- \frac{4}{3} + \d} } \notag\\
				&\leq \sum_{|k| \leq (N \D)^{-1}} \1_{\la_\a \in [\b_k, \b_{k+1})} \1_{\la_{\a+1} \in [\b_k, \b_{k+1} + N^{- \frac{4}{3}+\d} )} \notag\\
				&\leq \sum_{|k| \leq (N\D)^{-1}} \1_{\la_\a \in [\b_k, \b_{k+1})} \l[ q_{\a+1}(\Tr f_{\b_{k+1} + N^{ - \frac{4}{3} + \d}  + \h_0}) - q_{\a+1}(\Tr f_{\b_k - \h_0})  \r] \notag\\
				&\leq \sum_{|k| \leq (N\D)^{-1}} \l( q_\a(\Tr f_{\b_{k+1} + \h_0}) - q_\a(\Tr f_{\b_k - \h_0}) \r) \l[ q_{\a+1}(\Tr f_{\b_{k+1} + N^{ - \frac{4}{3} + \d}  + \h_0}) - q_{\a+1}(\Tr f_{\b_k - \h_0})  \r]\label{ineq1}.
		\end{align}
		Henceforth we use the shorthand $\bd q_\a(\b, \b') := q_\a(\Tr f_{\b'}) - q_\a(\Tr f_{\b})$, and define 
		\begin{equation}\label{eqn:Qalpha1}
				Q_{\alpha} := \sum_{|k| \leq (N\D)^{-1}} \bd q_\a\l(\b_k - \h_0, \b_{k+1} + \h_0\r) \bd q_{\a+1}\l(\b_k - \h_0, \b_{k+1} + N^{- \frac{4}{3} + \d}  + \h_0\r),
		\end{equation}
		Let $r \geq 2$ be even.  We apply Markov's inequality to the quantity in (\ref{extreme-gap-stability}) to get
		\begin{multline}
				\P^{\bm{v}}\l(\# \l\{ \ga N \leq \a \leq (1 - \ga) N : \la_{\a+1} - \la_\a \leq N^{- \frac{4}{3} + \d}  \r\} > N^{\kappa} \r) \\
				\leq N^{-r\kappa}\E^{\bm{v}} \l[ \l(\sum_{\ga N \leq \a \leq (1 - \ga)N} \1_{\la_{\a+1} - \la_\a \leq N^{- \frac{4}{3} + \d} }  \r)^r \r]  
				= N^{-r\kappa} \E^{\bm{v}} \l[  \l(\sum_\a Q_{\a}  \r)^r \r]. \label{egs-1}
		\end{multline}
		We now extend the definition (\ref{eqn:Qalpha1}) to 
				\begin{equation}\label{eqn:Qalpha1-ext}
				Q_{\alpha,\ell} := \sum_{|k| \leq (N\D)^{-1}} \bd q_\a\l(\b_k - \h_0, \b_{k+1} + \h_0\r) \bd q_{\a+1}\l(\b_k - \h_0, \b_{k+1} + N^{- \frac{4}{3} + \d}  + (4(r-\ell)+1)\h_0\r),
		\end{equation}
so that $Q_\alpha=Q_{\alpha,r}$.
		
		We will now prove by induction that there exists $p\geq 1$ such that  the following property $(P_\ell)$ holds for  all $0\leq \ell\leq r$: For any $c>0$ there exists $C=C(\ell,r)$ such that 
		\[\E^{\bm{v}}\Big[\big(\sum_\a Q_{\a,\ell} \big)^r\Big]\leq \varphi^C N^{r-\ell+c}.\]
		To prove ($P_0$), note that from (\ref{smoothing-est}) we can  bound $Q_\a$ with
		\begin{align*}
				&\sum_{|k| \leq (N\D)^{-1}} \l( q_\a(\Tr f_{\b_k - \h_0}) - q_\a(\Tr f_{\b_{k+1} + \h_0}) \r)\l[ q_{\a+1}(\Tr f_{\b_k - \h_0}) - q_{\a+1}(\Tr f_{\b_{k+1} + N^{ - \frac{4}{3} + \d}  + (4r+1)\h_0}) \r] \\
				&\leq \sum_{|k| \leq (N\D)^{-1}}   \1_{\la_\a \in [\b_k - 2\h_0, \b_{k+1} + 2\h_0)} \1_{\la_{\a+1} \in [\b_k - 2\h_0, \b_{k+1} + N^{- \frac{4}{3} + \d}  + (4r+2)\h_0)} \\
				&\leq 2 \1_{\la_{\a + 1} - \la_\a \leq N^{- \frac{4}{3} + \d}+\Delta  + (4r+4)\h_0},
		\end{align*}
		so $\sum Q_\alpha\leq 2N$ and $(P_0)$ follows. 
		
		We now assume ($P_\ell$) and 
change $ \E^{\bm{v}}$ with  $\E^{\bm{w}}$ by following the reasoning 
		 after (\ref{eqn:swapp}). 
		 The difference is that we now need to bound $\E[\|F^{(p+1)}\|_\infty]$ in (\ref{eqn:F}) for the new observable
		 $F=(G_{\ell+1})^r$ where 
		 $G_{\ell}=\sum_\a Q_{\a,\ell}.
		 $
We have
\[
|F^{(p+1)}|\leq C(r,p)\sum_{m=1}^{p+1}|G_{\ell+1}|^{r-m}\sum_{\sum \beta_i=m,\sum i\beta_i=p+1}\prod_i|\partial^iG_{\ell+1}|^{\beta_i}.
\]
Note the following elementary fact:  If $Q_{\alpha,\ell}\neq 1$ then $Q_{\alpha,\ell+1}$ is constant equal to $0$ on $|x|<N^{-\frac{1}{2}+\varepsilon}$; this is due to ${\rm d} \lambda_k=\langle u_k,{\rm d}H u_k\rangle$ so by eigenvector delocalization we have $|\lambda_k^{x}-\lambda_k^{0}|\leq N^{-\frac{3}{2}+\varepsilon}$, implying that $Q_\alpha=0$ constantly in $x$.
This implies that 
\begin{equation}\label{eqn:Stab2}
		\|\partial^iG_{\ell+1}\|_{\infty}\leq  |G_\ell|\cdot\max_{\sum c_\alpha=i}|\partial^{c_\alpha} Q_\alpha|.
\end{equation}
By the same proof as for (\ref{eqn:Stab}), we obtain 
		we have
			\begin{equation}\label{eqn:Stab2-b}
		\|\partial^iG_{\ell+1}\|_{\infty}\leq  \frac{\varphi^C}{(N\eta_d)^{i}}\cdot |G_\ell|.
		\end{equation}
		with overwhelming probability, so that
		\[
		\E[\|F^{(p+1)}\|_\infty]\leq C(r,p)\varphi^C\frac{1}{(N\eta_d)^{p+1}}\E[|G_\ell|^{r}].
		\]
		From the induction hypothesis ($P_\ell$) we have $\E[|G_\ell|^{r}]\leq\varphi^C N^{r-\ell+c}$ and we have obtained
	\[
		\E[\|F^{(p+1)}\|_\infty]\leq C(r,p)\varphi^C \frac{N^{r- \ell+c}}{(N\eta_d)^{p+1}},
		\]
		 so that
		\begin{equation}\label{egs-2}
			 \E^{\bm{v}} \l[  \l(\sum_\a Q_{\a,\ell+1}  \r)^r \r] = \E^{\bm{w}} \l[  \l(\sum_\a Q_{\a,\ell+1}  \r)^r \r] + {\rm O}\left( \varphi^C N^{2+r-\ell+c}\frac{N^{-(p+1)/2}}{(N\h_d)^{p+1}}\right).
		\end{equation}
		Then we can apply the relaxation estimate (\ref{e:gaps-relaxation}) so that for arbitrary $\iota > 0$, we have with overwhelming probability
		\begin{equation}
				\l|\l( \la_{\a + 1}^{\tb{w}} - \la_\a^{\tb{w}}\r) - \l( \la_{\a + 1}^{GUE} - \la_{\a}^{GUE} \r) \r| = O\l( \frac{N^{\iota}}{N^2}\r).
		\end{equation}
		Now, note that we can apply (\ref{smoothing-est}) again to bound $Q_{\a,\ell+1}$:
		\begin{align*}
				&\sum_{|k| \leq (N\D)^{-1}} \l( q_\a(\Tr f_{\b_k - \h_0}) - q_\a(\Tr f_{\b_{k+1} + \h_0}) \r)\l[ q_{\a+1}(\Tr f_{\b_k - \h_0}) - q_{\a+1}(\Tr f_{\b_{k+1} + N^{ - \frac{4}{3} + \d}  + (4(r-\ell-1)+1)\h_0 }) \r] \\
				&\leq \sum_{|k| \leq (N\D)^{-1}}   \1_{\la_\a \in [\b_k - 2\h_0, \b_{k+1} + 2\h_0)} \1_{\la_{\a+1} \in [\b_k - 2\h_0, \b_{k+1} + N^{- \frac{4}{3} + \d} + (4(r-\ell-1)+2)\h_0} \\
				&\leq 2 \1_{\la_{\a + 1} - \la_\a \leq \Delta+ N^{- \frac{4}{3} + \d}  +  (4(r-\ell-1)+4)\h_0}\\
				&\leq 2 \1_{\la_{\a + 1}^{\t{GUE}} - \la_{\a}^{\t{GUE}} \leq 2 N^{- \frac{4}{3} + \d} }.
		\end{align*}
		where we used that $\frac{N^{\iota}}{N^2}$, $\h_0$, and  $\D$ living on smaller scales than $N^{-4/3 + \d}$. In sum, we have proved
		\[
		\E^{\bm{v}}\Big[\big(\sum_\a Q_{\a,\ell+1} \big)^r\Big]
		\leq    \E^{\t{GUE}} \l[  \l( \sum_{\ga N \leq \a \leq (1-\ga)N} \1_{\la_{\a + 1} - \la_\a \leq 7 N^{- \frac{4}{3} + \d}} \r)^r \r]+ {\rm O}\left( \varphi^C N^{2+r-\ell+c}\frac{N^{-(p+1)/2}}{(N\h_d)^{p+1}}\right).
\]
		By the following Lemma~\ref{GUE-count-r-lma}, we can pick $\d$ sufficiently small, independent of $N$, so that the above expectation is ${\rm O}(N^c)$ for arbitrarily small $c$.  Moreover, for $p=\lfloor100/\varepsilon\rfloor$ (remember $\varepsilon$ from $\eta_d=N^{-\frac{3}{2}+\varepsilon}$) the above error term is clearly 
		 ${\rm O}\left( \varphi^C N^{r-(\ell+1)+c}\right)$, concluding the proof of $(P_{\ell+1})$.
		 
		The result of the lemma now follows easily from (\ref{egs-1}) and  the bound given by the property ($P_r$).
\end{proof}

\begin{lemma}\label{GUE-count-r-lma}
		Let $\varepsilon > 0$ be given and $r > 0$ be an arbitrary integer, independent of $N$.  There exists a $\d > 0$ sufficiently small and $N_0$ such that for $N > N_0$ we have 
		\begin{equation}\label{GUE-count-r}
			 \E^{{\rm GUE}} \l[  \l( \sum_{\ga N \leq \a \leq (1-\ga)N} \1_{\la_{\a + 1} - \la_\a \leq N^{- \frac{4}{3} + \d}} \r)^r \r] < N^{\varepsilon}.
		\end{equation}
\end{lemma}

\begin{proof}
		We proceed as follows., after expansion of the product. Via the determinantal representation of the spectrum of GUE and Hadamard's inequality, we can essentially treat products of indicators with no shared indices as independent. So the main difficulty will be to show that the total contribution of those terms with overlapping indices is asymptotically small. 

		First, from Lemma~\ref{GUE-small-gap-delta} we have 
\[
				\int \int \1_{|x_\a - x_\b| \leq N^{- \frac{4}{3} + \d}} d\rho_2^{\t{GUE}(N)}(x_\a, x_\b)\asymp \int_{0 \leq u \leq  N^{- \frac{4}{3} + \d}} \l[ N^{4} u^2 + O(N) \r] du \asymp N^{3\d} + N^{- \frac{1}{3} + \d}.
\]
		For the terms with shared indices, we can afford a cruder estimate, starting with
		\begin{multline}
				\int \prod_{i = 0}^{k-1} \1_{|x_i - x_{i+1}| \leq N^{- \frac{4}{3} + \d}} d\rho_{k+1}^{\t{GUE}(N)}(x_0, x_1, \ldots, x_k)  \\
				\leq \int_{-2 + \varepsilon_0}^{2 - \varepsilon_0} dx_0 \int \prod_{i=1}^k \1_{|x_0 - x_i| \leq k N^{- \frac{4}{3} + \d}} d\rho_{k+1}^{\t{GUE}(N)}(x_0, x_1, \ldots, x_k). \label{shared-integral}
		\end{multline}
		Let $K(x, y)$ be the determinantal kernel for the GUE point process. Then, due to the multilinearity of the determinant, we can write
		\begin{multline*}
				\rho_{k+1}^{\t{GUE}(N)}(x_0, x_1, \ldots, x_k) = \det_{(k+1)\x(k+1)} [K(x_i, x_j)]_{1 \leq i, j \leq k}\\
				= \begin{vmatrix}
						K(x_0, x_0) & K(x_0, x_1) - K(x_0, x_0) & \cdots & K(x_0, x_k) - K(x_0, x_0) \\
						K(x_1, x_0) & K(x_1, x_1) - K(x_1, x_0) &  \cdots & K(x_1, x_k) - K(x_1, x_0) \\
						\vdots & \vdots & \ddots & \vdots\\
						K(x_k, x_0) & K(x_k, x_1) - K(x_k, x_0) &  \cdots & K(x_k, x_k) - K(x_k, x_0)
				\end{vmatrix}.
		\end{multline*}
		From \cite{BABo13}*{Lemma 2.8} we know that the first column is $O(N)$ and the remaining $k$ columns are of order $O(N^2 \* k N^{- 4/3 + \d})$. Moreover, as the integration domain is $O( (k N^{-4/3 + \d})^k)$, so that (\ref{shared-integral}) is of order $O(k^{2k} N^{k(-2/3 + 2\d) + 1})$.

		Now, we handle the combinatorics of the product expansion in (\ref{GUE-count-r}) in the following way. Because each indicator is idempotent, we will split into cases the number $m$ of indicators represented in a term of the expansion. There are $\binom{r}{m}$ ways to choose the locations of the represented indicators, and at most $m^{r - m}$ ways to select which indicators the remaining $r-m$ are to agree with. Now, among the terms of $m$ different indicators, there are $\cal P(m)$, the number of partitions of $m$, different types of terms to consider. In particular, let $(p_1, p_2, \ldots, p_\ell)$ be a partition of $m$. Then, collectively, the terms of this structure are bounded via Hadamard's inequality and (\ref{shared-integral}) by (below we relabel the points $x_i$ as $x_{j,i}$ to emphasize the combinatorial structure)
		\begin{multline*}
				\int \prod_{j=1}^\ell \prod_{i=1}^{p_j} \1_{|x_{j,i} - x_{j,i+1}| \leq N^{- \frac{4}{3} + \d}} d\rho^{\t{GUE}(N)}_{m + \ell}(\{x_{j,i}\}) 
				\leq \prod_{j=1}^\ell \l[ \int \prod_{i=1}^{p_j}\1_{|x_{j,i} - x_{j,i+1}| \leq N^{- \frac{4}{3} + \d}} d\rho_{p_j + 1}^{\t{GUE}(N)}(\{x_{j,i}\}_{i=1}^{p_j + 1} )\r] \\
				\leq \prod_{j=1}^\ell \l[ O \l( p_j^{2p_j} N^{p_j (-2/3 + 2\d) + 1}  \r)\1_{p_j \geq 2} + O \l(  N^{3 \d} ) \r) \1_{p_j = 1}\r] 
				\leq O \l( m^{2m} N^{m (-2/3 + 2\d) + m/2} + N^{m 3 \d} \r). \label{extend-loc}
		\end{multline*}
		Altogether we  bound the left hand side of (\ref{GUE-count-r}) by 
		\begin{multline*}
				 \E^{\t{GUE}} \l[  \l( \sum_{\ga N \leq \a \leq (1-\ga)N} \1_{\la_{\a + 1} - \la_\a \leq N^{- \frac{4}{3} + \d}} \r)^r \r] \\
				\leq \sum_{m=1}^{r} \binom{r}{m} m^{r-m} \cal P(m) O \l( m^{2m} N^{m (-2/3 + 2\d) + m/2} + N^{3m \d} \r) 
				\leq C  \* N^{3r \d},
		\end{multline*}
		where $C = C(r)$ is a constant depending only on $r$, which is independent of $N$. Therefore, picking $\d < \varepsilon/(3r)$ yields (\ref{GUE-count-r}).
\end{proof}

In the same way as Theorem~\ref{t:gap-rate} was proven from Theorem~\ref{4-moment-thmbis}, Theorem~\ref{t:min-gap-rate} follows immediately from the following Theorem~\ref{4-moment-thmbis-min}.

\begin{theorem}\label{4-moment-thmbis-min}
		There exists $\varepsilon_0>0$ such that for any $0<\varepsilon<\varepsilon_0$ and $0 < \ga < 1/2$ the following holds. There exists a $p = p(\varepsilon)> 0$ and $N_0 = N_0(D, \varepsilon) > 0$ both independent of $N$, such that if 
 the $\bf{v}$ and $\bf{w}$ matrix entries have matching moments up to the $p$-th order, and that $\bf{w}$ is a Gaussian-divisible distribution as in (\ref{GDE-form}) with $t\asymp 1$,  then for all $N \geq N_0$  and $0 < u < N^\varepsilon$
		\begin{equation}\label{extreme-gap-mm}
				\l[ \P^{\bm{v}} - \P^{\bm{w}} \r] \l(  N^{4/3}\min_{\ga N \leq \a \leq (1 - \ga) N} \l( \la_{\a+1} - \la_\a \r) \geq u \r) =O (N^{-1/2 + \varepsilon}).
		\end{equation}
		Moreover, the same rate of convergence holds if we replace $\bf{w}$ with the GUE.
\end{theorem}

\begin{proof}
Similarly to (\ref{eqn:Qalpha1-ext}) we denote $\bd q_\a(\b, \b') := q_\a(\Tr f_{\b'}) - q_\a(\Tr f_{\b})$, and define 
		\begin{equation}\label{eqn:Qalpha2}
				Q_\a := \sum_{|k| \leq (N\D)^{-1}} \bd q_\a\l(\b_k - \h_0, \b_{k+1} + \h_0\r) \bd q_{\a+1}\l(\b_k - \h_0, \b_{k+1} + N^{- \frac{4}{3}}u  + \h_0\r). 
		\end{equation}
Then
\[
\1_{N^{4/3}\min_{\ga N \leq \a \leq (1 - \ga) N} \l( \la_{\a+1} - \la_\a \r) \geq u}=\prod_{\alpha}\1_{ ( \la_{\a+1} - \la_\a ) \geq N^{-4/3}u}\geq \prod_{\alpha}(1-Q_\alpha)=:F
\]
and we now want to bound $\|F^{(p+1)}\|_\infty$ to estimate the error between the ensembles $\pmb{v}$ and $\pmb{w}$ by Lindeberg exchange, as in (\ref{eqn:swapp}).  Note that if $\lambda^{\pmb{v}}_{\alpha+1}-\lambda^{\pmb{v}}_\alpha>N^{-\frac{4}{3}+\delta}$  then with overwhelming probability we have 
$\|\partial^\ell Q_\alpha\|_\infty=0$ for all $1\leq \ell\leq p$ (this is due to ${\rm d} \lambda_k=\langle u_k,{\rm d}H u_k\rangle$ so by eigenvector delocalization we have $\|\lambda_k^{x}-\lambda_k^{0}\|_\infty\leq N^{-\frac{3}{2}+3\varepsilon}$, implying that $Q_\alpha=0$ constantly in $x$).

Together with Lemma \ref{eqn:notmany} this implies that for any fixed $\kappa>0$, with overwhelming probability we have
\[
\|F^{(p+1)}\|_\infty \leq (N^{\kappa})^{p+1}
\max_{\sum c_\alpha=p+1}\prod_\alpha|\partial^{c_\alpha} Q_\alpha|.
\]
By the same proof as for (\ref{eqn:Stab}), we obtain 
			\begin{equation}\label{eqn:Stab3}
\|F^{(p+1)}\|_\infty \leq (N^{\kappa})^{p+1}\frac{\varphi^C}{(N\eta_d)^{p+1}},
		\end{equation}
		and the moment matching yields, for arbitrary small $c>0$,
		\[
		\P^{\bm{v}}  \l(  N^{4/3}\min_{\ga N \leq \a \leq (1 - \ga) N} \l( \la_{\a+1} - \la_\a \r) \geq u \r)\geq \E^{\bm{w}}[F]+{\rm O}\left(\frac{N^{2+c}}{(N\eta_d)^{p+1}}\right) 
		\]
Using $Q_\alpha\leq \1_{\lambda_{\alpha+1}-\lambda_{\alpha}\leq N^{-4/3}u+10\eta_0}$ we obtain
\begin{multline*}
		\P^{\bm{v}}  \l(  N^{4/3}\min_{\ga N \leq \a \leq (1 - \ga) N} \l( \la_{\a+1} - \la_\a \r) \geq u \r) \\
		\geq
				\P^{\bm{w}}  \l(\min_{\ga N \leq \a \leq (1 - \ga) N} \l( \la_{\a+1} - \la_\a \r) \geq N^{-\frac{4}{3}}u+ 10\eta_0\r)
		+{\rm O}\left(\frac{N^{2+c}}{(N\eta_d)^{p+1}}\right) 
\end{multline*}
		With 
		(\ref{e:gaps-relaxation}) this gives
		\begin{multline*}
		\P^{\bm{v}}  \l(  N^{4/3}\min_{\ga N \leq \a \leq (1 - \ga) N} \l( \la_{\a+1} - \la_\a \r) \geq u \r) \\
		\geq
				\P^{{\rm GUE}}  \l(\min_{\ga N \leq \a \leq (1 - \ga) N} \l( \la_{\a+1} - \la_\a \r) \geq N^{-\frac{4}{3}}u+ 10\eta_0+\frac{N^c}{N^2}\r)
		+{\rm O}\left(\frac{N^{2+c}}{(N\eta_d)^{p+1}}\right).
		\end{multline*}
		One then concludes easily from the facts that (1) $\eta_0$ and $N^{c-2}$ live on a scale smaller than $N^{-4/3}u$,  (2) for GUE $N^{4/3}\min_{\ga N \leq \a \leq (1 - \ga) N} \l( \la_{\a+1} - \la_\a \r)$ converges in distribution to a random variable with density,  (3) the reverse inequality also holds with a similar reasoning.
\end{proof}

\begin{remark}
		The proof of Theorem~\ref{4-moment-thmbis-min} may be carried out for the real case, to obtain (\ref{HN-GOE-min}). To do this, one would need to extend Lemma~\ref{GUE-count-r-lma} for the GOE,  with technical input being correlation inequalities of type $\rho_k^{\t{GOE}(N)}(x_1, \ldots, x_{2k})\lesssim\rho_2^{\t{GOE}(N)}(x_1,x_2)\dots\rho_2^{\t{GOE}(N)}(x_{2k-1},x_{2k})$,  based on the Pfaffian representation of the correlation functions.  

	The estimate (\ref{HN-GOE-min}) corresponds to scales slightly above the expected scale $N^{-3/2}$ of the smallest gap.  This is because our moment matching estimates of type  (\ref{eqn:MomentMatch}) is only strong at a scale $\h = N^{-3/2 + \d}$, hence the additional factor of $N^\d$ in (\ref{HN-GOE-min}). 
\end{remark}

\appendix
\section{Appendix}

The following two classical results from \cite{ErdYauYin2012} are stated for ease of reference.
\begin{theorem}[Strong local semicircle law]\label{SLSL}
		Let $H = H_N$ be a real or complex generalized Wigner ensemble satisfying the exponential decay condition (\ref{a:exp-decay}). Let $|E| \leq 10$ and $0 < \h \leq 10$. Then for all $D > 0$ there exists an $N_0 > 0$ such that for all $N \geq N_0$ we have
		\begin{equation}
				\P\l( \max_{i,j} |G_{ij}(z) - \d_{ij} m_{sc}(z)| > \vphi^C \l( \sqrt{ \frac{\Im m_{sc}(z)}{N\h} } + \frac{1}{N\h} \r) \r) \leq N^{-D}.
		\end{equation}
\end{theorem}

		\begin{remark}
		The above local law also holds for any matrix of type $Q+x D$ with $|x|<N^{-1/2+\kappa}$ ($\kappa<1/2$) as defined in (\ref{qa}),   with the error term $\varphi^C$ with a different $C$.  Moreover the error bound is uniform in $z$:
				\begin{equation}\label{eqn:Extend}
				\P\l( \max_{i,j,|E| \leq 10,  0 < \h \leq 10, |x|\leq N^{-\frac{1}{2}+\kappa}} \left|\left(\frac{1}{Q-z+x D}\right)_{ij} - \d_{ij} m_{sc}(z)\right| > \varphi^C \l( \sqrt{ \frac{\Im m_{sc}(z)}{N\h} } + \frac{1}{N\h} \r) \r) \leq N^{-D}.
		\end{equation}
		Indeed e.g. in the bulk the above local law follows from a simple resolvent expansion when $\eta>N^{-1}$ as in  (\ref{eqn:expansion}) (this is where the hypothesis $\kappa<1/2$ is essential), to compare $Q+xD$ to $Q+\frac{v_{ab}}{\sqrt{N}}D$, for which the local law is known thanks to  
		Theorem \ref{SLSL}.  Then monotony allows to extend the result from  to any $\eta>0$ (see e.g. \cite{BK}*{Lemma 6.2}).
		  Uniformity in $z$ follows from a simple grid argument (see e.g. \cite{BK}*{Remark 2.7}).
		\end{remark}

\begin{theorem}[Rigidity of eigenvalues]\label{rigidity}
		Under the same assumptions as Theorem~\ref{SLSL}, we have
		\begin{equation}\label{eqn:rigidity}
				|\la_k - \ga_k| \leq \vphi^C N^{- \frac{2}{3}} (\hat k)^{ - \frac{1}{3}} \t{ for all $1 \leq k \leq N$} 
		\end{equation}
		with overwhelming probability for all sufficiently large $N$.
\end{theorem}

\begin{remark} From (\ref{eqn:Extend}) and (\ref{eqn:rigidity}) one obtains the following easy estimates  by spectral decomposition. 

First,  denoting $\pmb{\lambda}^{(x)}$ the eigenvalues of $Q+xD$ and $\pmb{u}^{(x)}$ its eigenvectors for any $c>0,D>0$ there exists $C>0$ such that for all $N$ we have
\begin{equation}\label{eqn:far}
\mathbb{P}\left(\sup_{i,j,{\rm dist}(z,[-2,2])>c, |x|\leq N^{-\frac{1}{2}+\kappa}}|\left(\frac{1}{Q-z+xD}\right)_{ij}|\geq \varphi^C\right)\leq N^{-D}.
\end{equation}
Indeed one can simply write
\[|\left(\frac{1}{Q-z+xD}\right)_{ij}|
\leq \sum_{k=1}^N\frac{|u^{(x)}_k(i)u^{(x)}_k(j)|}{|\lambda^{(x)}_k-z|}\leq \varphi^C.
\]
because for any $k$ we have $|u^{(x)}_k|_\infty\le \varphi^C N^{-1/2}$ (from (\ref{eqn:Extend})) and $|\lambda^{(x)}_k|\leq 2+\kappa/2$ (from Theorem $\ref{rigidity}$ and the triangle inequality on spectral norms).

Second,  by the same reasoning we have
\begin{equation}\label{eqn:far2}
\mathbb{P}\left(\sup_{i,{\rm dist}(z,[-2,2])>c, |x|\leq N^{-\frac{1}{2}+\kappa}}|\frac{1}{\eta}{\rm Im}\left(\frac{1}{Q-z+xD}\right)_{ii}|\geq \varphi^C\right)\leq N^{-D}.
\end{equation}
\end{remark}

\begin{lemma}\label{smoothing-approx}
		Uniformly in  $E\in[-2+c,2-c]$ the following holds,  where $c>0$ is a fixed small constant.  Let $\h_0 = N^\varepsilon \h_d $, where $\varepsilon > 0$, so that $\h_d \ll \h_0 \ll N^{-1}$. 
		With the definitions of $q_\alpha$ and $f_E$ as in  (\ref{qa}) and above,
for large enough $N$ we have the following deterministic estimates:
		\begin{equation}\label{smoothing-est-0}
				\Tr f_{E-\h_0}(H) - e^{-N^{\varepsilon/2}} \leq \Nor(E) \leq \Tr f_{E + \h_0}(H) +  e^{-N^{\varepsilon/2}},
		\end{equation}
		and
		\begin{equation}\label{smoothing-est}
				q_\a(\Tr f_{E - \h_0}(H)) \leq \1_{\Nor(E) \geq \a} \leq q_\a(\Tr f_{E + \h_0}(H)) .
		\end{equation}
\end{lemma}

\begin{proof} We have
\[
\Tr f_{E-\h_0}(H)\leq \Nor(E) + N\sup_{x>E} f_{E-\h_0}(x)= \Nor(E) +\frac{N}{1+e^{\frac{\eta_0}{\eta_d}}}.
\]
		This proves the first inequality of (\ref{smoothing-est-0}), and the second inequality is obtained similarly. The subsequent bound (\ref{smoothing-est}) follows immediately from the definition of $q_\a$ and the fact that $\Nor(E)$ is integer valued.
\end{proof}

The following estimate on the 2-point correlation function for the GUE follows immediately from the proof of \cite{BABo13}*{Lemma 2.9}, with only notational changes to extend $u$ to the scale $cN^{- \frac{4}{3} + \d}$.

\begin{lemma}\label{GUE-small-gap-delta}
		Fix small constants $\varepsilon_0 > 0, c > 0$ and a small enough  $\d > 0$. Then  as $N \to \infty$, uniformly for $x \in (-2 + \varepsilon_0, 2 - \varepsilon_0), |u| < c N^{- \frac{4}{3} + \d}$,  we have
		\begin{equation}
				\rho_2^{{\rm GUE}(N)}(x, x + u) = \frac{1}{48\pi^2} N^{4}(4 - x^2)^2 u^2 + O(N).
		\end{equation}
\end{lemma}

\end{document}